\documentclass{amsart}
\usepackage[margin=3.5cm]{geometry}
\usepackage[dvipsnames]{xcolor}

\usepackage{amsfonts, amsthm, amssymb,verbatim,amscd,amsmath, blindtext}

\usepackage{multirow,multicol}
\usepackage{graphicx}
\graphicspath{ {./images/} }
\usepackage{enumitem,yhmath}
\usepackage{amssymb}
\usepackage{ytableau}
\usepackage{tikz}
\usepackage{tikz-cd,rotating}
\usepackage{float, pifont}
\usepackage[utf8]{inputenc}
\usepackage{mathtools,nccmath}
\usepackage{float}
\usepackage{adjustbox}
\usepackage[colorlinks=true,citecolor=blue,linkcolor=BrickRed]{hyperref}
\usepackage{caption} 
\newtheorem{theorem}{Theorem}[section]
\newtheorem{lemma}[theorem]{Lemma}
\newtheorem{corollary}[theorem]{Corollary}
\newtheorem{proposition}[theorem]{Proposition}

\theoremstyle{remark}

\newtheorem{question}[theorem]{Question}

\theoremstyle{definition}

\newtheorem{example}[theorem]{Example}

\DeclareMathOperator{\reg}{reg}

\DeclareMathOperator{\Ass}{Ass}
\DeclareMathOperator{\Min}{Min}
\DeclareMathOperator{\supp}{supp}

\DeclareMathOperator{\vv}{v}

\def\m{\mathfrak{m}}

\def\k{\mathrm{k}}

\title{Comparing v-numbers of symbolic and ordinary powers of squarefree monomial ideals}



\author[Chau]{Trung Chau}
\address{Chennai Mathematical Institute, H1 SIPCOT IT Park, Siruseri, Tamil Nadu 603103. India}
\address{Department of Mathematics, University of Manitoba, 420 Machray Hall, 186 Dysart
Road, Winnipeg, MB R3T 2N2 Canada}
\email{chauchitrung1996@gmail.com}

\author{T\`ai Huy H\`a}
\address{Tulane University, Mathematics Department, 6823 St. Charles Avenue, New Orleans, LA 70118, USA}
\email{tha@tulane.edu}

\author{A. V. Jayanthan}
\address{Department of Mathematics, Indian Institute of Technology Madras, Chennai, Tamil Nadu, India - 600036.}
\email{jayanav@iitm.ac.in}

\author{Thanh Vu}
\address{Institute of Mathematics, VAST, 18 Hoang Quoc Viet, Hanoi, Vietnam}
\email{vuqthanh@gmail.com}

\begin{document}

\begin{abstract} Let $I$ be a squarefree monomial ideal, and let $d$ denote the maximum degree of a minimal generator of $I$. We prove that
\[
v(I^t)\le v(I^{(t)})+(t-1)(d-1)
\]
for all $t\ge1$, where $I^{(t)}$ denotes the $t$-th symbolic power of $I$. In particular, this bound does not depend on the number of variables in the ambient polynomial ring. On the other hand, for every fixed exponent $t\ge2$, the difference
\[
v(I^{(t)})-v(I^t)
\]
can be arbitrarily large. Finally, we determine the $v$-numbers of the ordinary and symbolic powers of edge ideals of paths and cycles.
\end{abstract}

\maketitle

\section{Introduction} 
Let $S = \k[x_1, \ldots, x_n]$ be a standard graded polynomial ring over a field $\k$, and let $I$ be a non-zero homogeneous ideal of $S$. For an associated prime $P$ of $I$, the local $v$-number of $I$ at $P$ is defined by 
$$v_P(I) = \min \{d \ge 0 \mid \exists f \in S_d \text{ such that } I : f = P\}.$$
The $v$-number of $I$ is defined by
\[
v(I)=\min_{P\in\Ass(I)}v_P(I),
\]
where $\Ass(I)$ denotes the set of associated primes of $I$. The $v$-number was introduced and studied by Cooper, Seceleanu, Tohaneanu, Vaz Pinto, and Villarreal \cite{CSTVV} in their study of the minimum distance function of projective Reed--Muller-type codes. From the definition, we see that it is closely related to the associated primes of $I$, and hence is also of significant interest from the perspective of commutative algebra.

Conca \cite{C} and, independently, Ficarra and Sgroi \cite{FS} proved that
\[
v(I^t)=\alpha(I)t+b(I)
\]
for some constant $b(I)$ and for all sufficiently large $t$, where $\alpha(I)$ denotes the initial degree of $I$. For edge ideals, the asymptotic behavior of the $v$-number of $I(G)$ is completely determined by the results of Biswas, Mandal, and Saha \cite{BMS} and Ficarra and Marques \cite{FM}. Moreover, Kumar, Nanduri, and Saha \cite{KNS} proved that
\[
\lim_{t\to\infty}\frac{v(I^{(t)})}{t}=\widehat{\alpha}(I),
\]
where $\widehat{\alpha}(I)$ denotes the Waldschmidt constant of $I$.

Comparing the ordinary and symbolic powers of an ideal is a classical problem in commutative algebra. In this paper, we compare the $v$-numbers of ordinary and symbolic powers of squarefree monomial ideals, with particular emphasis on the case of edge ideals of graphs. In view of the asymptotic results mentioned above, one naturally expects that
\[
v(I(G)^{(t)})\le v(I(G)^t)
\]
for an arbitrary graph $G$ and every exponent $t$. However, for a fixed exponent $t$, the situation is considerably more subtle. Our main results are as follows.

\begin{enumerate}
    \item Let $I$ be a squarefree monomial ideal, and let $d$ denote the maximum degree of a minimal generator of $I$. For every associated prime $P$ of $I$, we prove that
    \[
    v_P(I^{(t)}) \le v_P(I^t) \le v_P(I^{(t)})+(t-1)(d-1).
    \]
    In particular,
    \[
    v(I^t)\le v(I^{(t)})+(t-1)(d-1)
    \]
    for all $t\ge2$. Furthermore, when $t=2$ and edge ideals of graphs, we establish the stronger inequality
    \[
    v(I(G)^2)\le v(I(G)^{(2)}).
    \]

    \item For every fixed integer $t\ge2$, we show that the difference
    \[
    v(I^{(t)})-v(I^t)
    \]
    can be arbitrarily large, even for edge ideals of graphs.

    \item We determine explicit formulas for the $v$-numbers of ordinary and symbolic powers of edge ideals of paths and cycles. In particular, our results for paths resolve \cite[Conjecture 4.2]{YHC}.
\end{enumerate}

The comparison results are proved in Section~\ref{sec_compare}. In Section~\ref{sec_exm}, we construct examples showing that, for every fixed integer $t\ge2$, the difference
\(v(I^{(t)})-v(I^t)
\)
can be arbitrarily large. We then determine the $v$-numbers of the ordinary and symbolic powers of edge ideals of paths and~cycles.

\medskip

\noindent\textbf{Acknowledgment.} Chau appreciates the support by the Infosys Foundation during his postdoc at Chennai Mathematical Institute, and the Pacific Institute for the Mathematical Sciences. Chau thanks H\`a for funding his visit to Tulane University in January 2026, and the International Center for Research and Postgraduate Training in Mathematics (ICRTM) for funding his visit to the Institute of Mathematics, Vietnam Academy of Science and Technology (VAST) in June-July 2026. Part of the work was done when Jayanthan was visiting Vietnam Institute for Advanced Study in Mathematics (VIASM). He wishes to thank H\`{a} for the invitation and VIASM for the local hospitality. His travel was funded by the MATRICS Grant (MTR/2023/000335). H\`a is partially supported by a Simons Foundation~grant.

\section{Local $v$-numbers of powers and symbolic powers}\label{sec_compare}
Throughout the paper, let $S=k[x_1,\ldots,x_n]$ be a standard graded polynomial ring over a field $k$. For a nonzero homogeneous ideal $I$ of $S$, we denote by $\Ass(I)$ the set of associated primes of $I$.

When $I$ is a squarefree monomial ideal, every associated prime of $I$ is minimal. Moreover, each associated prime $P$ of $I$ is generated by a subset of the variables. For a monomial $f$, we denote by $\deg(f)$ its total degree and, for each $i\in[n]$, by $\deg_i(f)$ the exponent of $x_i$ in $f$. We further define
\[
\deg_P(f)=\sum_{x_i\in P}\deg_i(f).
\]
The $t$-th symbolic power of $I$ is
\[
I^{(t)}=\bigcap_{Q\in\Ass(I)}Q^t.
\]

We first establish some general properties and show that the computation of the $v$-numbers of symbolic powers of squarefree monomial ideals can be formulated as an integer linear programming problem.

\begin{lemma}\label{lem_deg}
Let $I\subseteq S$ be a nonzero squarefree monomial ideal, let $P$ be an associated prime of $I$, and let $t\ge1$ be an integer. Let $J$ be a monomial ideal satisfying
\( I^t\subseteq J\subseteq I^{(t)},
\)
and let $u$ be a monomial such that
\(
J:u=P.
\)
Then
\[
\deg_P(u)=t-1.
\]
\end{lemma}
\begin{proof}
Since the localizations of both $I^t$ and $I^{(t)}$ at $P$ are equal to $P^tS_P$, localizing the equality
\(J:u=P
\)
at $P$ yields
\[
P^tS_P:u=PS_P.
\]
Because the variables outside $P$ become units in $S_P$, it follows that
\[
\deg_P(u)=t-1,
\]
as required.
\end{proof}

\begin{lemma}\label{lem_colon}
Let $I\subseteq S$ be a nonzero squarefree monomial ideal, let $P$ be an associated prime of $I$, and let $t\ge1$ be an integer. Let $J$ be a monomial ideal satisfying
\(
I^t\subseteq J\subseteq I^{(t)},
\)
and let $u$ be a monomial such that
\[
\deg_P(u)=t-1
\quad \text{and} \quad J:u\supseteq P.
\]
Then
\[
J:u=P.
\]
\end{lemma}

\begin{proof}
Let $f$ be a monomial with $f\notin P$. Then \(
\deg_P(uf)=t-1,
\)
and hence
\(
uf\notin P^t.
\)
Since
\(
J\subseteq I^{(t)}\subseteq P^t,
\)
it follows that
\(
uf\notin J.
\)
Therefore,
\(
J:u=P.
\)
\end{proof}

\begin{lemma}\label{lem_criterion}
Let $I\subseteq S$ be a nonzero squarefree monomial ideal, let $P$ be an associated prime of $I$, and let $t\ge1$ be an integer. Let $u$ be a nonzero monomial in $S$. Then
\[
I^{(t)}:u=P
\]
if and only if
\[
\deg_P(u)=t-1
\quad\text{and}\quad
\deg_Q(u)\ge t
\]
for every associated prime $Q\neq P$ of $I$.
\end{lemma}

\begin{proof}
First, assume that
\[
I^{(t)}:u=P.
\]
By Lemma~\ref{lem_deg}, we have
\(
\deg_P(u)=t-1.
\) Now let $Q\neq P$ be another associated prime of $I$. Since $P$ and $Q$ are incomparable, there exists a variable $x$ such that $x\in P$ but $x\notin Q$. Since
\(
xu\in I^{(t)}\subseteq Q^t
\)
and $x\notin Q$, it follows that
\(
u\in Q^t.
\)
In particular,
\(
\deg_Q(u)\ge t.
\)

Conversely, assume that $u$ is a monomial such that
\[
\deg_P(u)=t-1
\quad\text{and}\quad
\deg_Q(u)\ge t
\]
for every associated prime $Q\neq P$ of $I$. Then $u\notin P^t$, and hence
\( 
u\notin I^{(t)}.
\) Now let $x\in P$ be a variable. Since
\(
\deg_P(xu)\ge t,
\)
we have
\(
xu\in P^t.
\)
Moreover, since
\(
u\in Q^t
\)
for every associated prime $Q\neq P$ of $I$, we also have
\(
xu\in Q^t
\)
for every such $Q$. Therefore,
\(
xu\in I^{(t)}.
\) By Lemma~\ref{lem_colon},
\[
I^{(t)}:u=P,
\]
as desired.
\end{proof}

Now let $I$ be a squarefree monomial ideal with associated primes
\[
P_i=(x_j\mid j\in A_i),
\]
where $A_i\subseteq [n]$. Lemma~\ref{lem_criterion} shows that the $v$-number of $I^{(t)}$ at $P_i$ can be computed by solving the following integer linear program:
\[
v_{P_i}(I^{(t)})
=
\min\left\{
\sum_{i=1}^n a_i
\;\middle|\;
\mathbf{a}\in\mathbb{N}^n,\;
\sum_{\ell\in A_i} a_\ell=t-1,\;
\sum_{\ell\in A_j} a_\ell\ge t
\text{ for all } j\neq i
\right\}.
\]
We now prove our first comparison result, showing that the local $v$-number of a symbolic power is at most that of the corresponding ordinary power at each minimal prime.
\begin{proposition}\label{prop_com1}
Let $I\subseteq S=k[x_1,\ldots,x_n]$ be a squarefree monomial ideal, let $P\in\Min(I)$, and let $t\ge1$. If $u$ is a monomial such that
\(I^t:u=P,
\)
then
\(I^{(t)}:u=P.
\)
Consequently,
\[
v_P(I^{(t)})\le v_P(I^t).
\]
\end{proposition}

\begin{proof}
By Lemma~\ref{lem_deg}, we have
\( 
\deg_P(u)=t-1.
\)
In particular, $u\notin P^t$, and hence $u\notin I^{(t)}$. Since
\(
I^t\subseteq I^{(t)},
\)
it follows that
\[
P=I^t:u\subseteq I^{(t)}:u.
\]
By Lemma~\ref{lem_colon},
\[
I^{(t)}:u=P.
\]

Finally, let $u$ be a monomial of minimum degree satisfying
\(I^t:u=P.
\)
Since the same monomial also satisfies
\(
I^{(t)}:u=P,
\)
we obtain
\[
v_P(I^{(t)})\le v_P(I^t),
\]
as desired.
\end{proof}

Recall that the edge ideal of a simple graph $G$ with vertex set
\( 
V(G)=[n]=\{1,\ldots,n\}
\) 
and edge set $E(G)$ is defined by
\[
I(G)=\bigl(x_ix_j \mid \{i,j\}\in E(G)\bigr)\subseteq S.
\]

We also use a characterization of symbolic powers of squarefree monomial ideals in terms of differential powers \cite[Lemma~2.6]{MNPTV}. Recall that, for monomials
$
f=x_1^{a_1}\cdots x_n^{a_n}
$ 
and $g$, the notation
\[
\frac{\partial^*(g)}{\partial^*(f)}
\]
denotes the $*$-partial derivative of $g$ with respect to $f$, that is, the partial derivative computed without coefficients. We now prove the following result.

\begin{theorem}\label{prop_com2}
Let $I$ be a squarefree monomial ideal, and let $d$ denote the maximum degree of a minimal generator of $I$. Then, for every associated prime $P$ of $I$ and every integer $t\ge1$, we have
\[
v_P(I^t)\le v_P(I^{(t)})+(t-1)(d-1).
\]
In particular, if $G$ is a graph, then
\[
v_P(I(G)^t)\le v_P(I(G)^{(t)})+t-1.
\]
\end{theorem}

\begin{proof}
Let $u$ be a monomial such that
\(
I^{(t)}:u=P.
\)
By Lemma~\ref{lem_deg}, we have
\[
\deg_P(u)=t-1.
\]
Write
\(
u=fg,
\)
where the support of $f$ is contained in $P$ and the support of $g$ is disjoint from $P$. Let $x_i$ be any variable in $P$. Since
\(
x_iu\in I^{(t)},
\)
it follows from \cite[Lemma~2.6]{MNPTV} that
\[
\frac{\partial^*(x_iu)}{\partial^*(f)}=x_ig\in I.
\]
Hence, there exists a monomial $v_i\mid g$ such that $x_iv_i$ is a minimal generator of $I$. In particular,
\[
\deg(v_i)\le d-1.
\]
Now let $v$ be the product of the monomials $v_i$, where each $v_i$ corresponds to a variable appearing in $f$. Since $\deg_P(f)=t-1$, there are at most $t-1$ such variables, and hence
\[
fv\in I^{t-1}.
\]
We claim that
\[
I^t:(uv)=P.
\]
Indeed, let $x_i\in P$. Then
$ 
x_iuv=(x_ig)(fv)\in I^t,
$ 
which shows that
$
P\subseteq I^t:(uv).
$ 
By Lemma~\ref{lem_colon},
\[
I^t:(uv)=P.
\]

Since
$
\deg(uv)\le\deg(u)+(t-1)(d-1),
$
we conclude that
\[
v_P(I^t)\le v_P(I^{(t)})+(t-1)(d-1),
\]
as desired. When $G$ is a graph, every minimal generator of $I(G)$ has degree $2$. Hence, the corresponding inequality for edge ideals follows immediately.
\end{proof}

As a consequence, we obtain the following corollary.

\begin{corollary} Let $I$ be a squarefree monomial ideal, and let $d$ denote the maximum degree of a minimal generator of $I$. Then, for all $t\ge 1$,
\[
v(I^t)\le v(I^{(t)})+(t-1)(d-1).
\]
\end{corollary}

\begin{proof}
Since $\Ass(I^{(t)})=\Ass(I)$, the conclusion follows immediately from Theorem~\ref{prop_com2} and the definition of the $v$-number.
\end{proof}

For the second power of an edge ideal, we obtain a stronger inequality.

\begin{theorem}\label{lem_min_sym_2}
Let $G$ be a simple graph, and let $P$ be a minimal prime of $I(G)$. Assume that
\( I(G)^{(2)}:u=P.
\)
Then there exists a monomial $f$ such that $\deg(f)\le \deg(u)$ and $I(G)^2:f$ is an associated prime of $I(G)^2$. Consequently,
\[
v(I(G)^2)\le v(I(G)^{(2)}).
\]
\end{theorem}

\begin{proof}
By \cite[Lemma 3.1]{MNPTV}, we have
\[
\sqrt{I^2:u}=\sqrt{I^{(2)}:u}=P.
\]
If $I^2:u=P$, then there is nothing to prove. Hence, we may assume that $I^2:u\neq P$. This implies that there exists a variable $x_1\in P$ such that
\[
x_1\in I^{(2)}:u
\quad\text{but}\quad
x_1\notin I^2:u.
\]
Equivalently, there exists a minimal generator $x_1x_2x_3$ of $I^{(2)}$ such that
$ 
\displaystyle x_1=\frac{x_1x_2x_3}{\gcd(x_1x_2x_3,u)}.
$ 
In particular,
$ 
u=x_2x_3g
$ 
for some monomial $g$ such that $x_1\nmid g$. Since $P$ corresponds to a minimal vertex cover of $G$, at least one of $x_2$ and $x_3$ belongs to $P$. Without loss of generality, we may assume that $x_2\in P$.

It follows that there exists a minimal generator $h$ of $I^{(2)}$ such that
$
\displaystyle x_2=\frac{h}{\gcd(h,u)}.
$ 
Since $u=x_2x_3g$, we must have $x_2^2\mid h$. Thus,
$ 
h=x_2ax_2b
$ 
for some variables $a$ and $b$. We claim that both $a$ and $b$ must be equal to $x_3$. Indeed, suppose that one of them, say $a$, is not equal to $x_3$. Then $a\mid u$. Consequently, $u$ is divisible by $ax_2x_3$, and hence $x_1u$ is divisible by
$ 
(x_2a)(x_1x_3).
$ 
Therefore,
$ 
x_1\in I^2:u,
$
which is a contradiction. Hence, $a=b=x_3$, and consequently,
$
u=x_2x_3^2v,
$
where $g=x_3v$. We next claim that
\[
N(\supp(v))\cap\{1,2,3\}=\emptyset.
\]
Indeed, if $1\in N(\supp(v))$, then $x_1u\in I^2$, contradicting the choice of $x_1$. Likewise, if $2\in N(\supp(v))$, then $x_1u\in I^2$, since it is divisible by
$ 
(x_1x_3)(x_2v).
$
Finally, if $3\in N(\supp(v))$, then $u\in I^2$, again a contradiction. Hence,
\[
N(\supp(v))\cap\{1,2,3\}=\emptyset.
\]
Now let
$ 
f=x_1x_2x_3v.
$ 
Then $f\notin I^2$. We claim that
\begin{equation}\label{eq_colon_triangle}    
I^2:(x_1x_2x_3)
=
I+(x_j\mid j\in N[\{1,2,3\}]).
\end{equation}
The inclusion
\[
I+(x_j\mid j\in N[\{1,2,3\}])
\subseteq I^2:(x_1x_2x_3)
\]
is clear. For the reverse inclusion, let $w$ be a monomial whose
support is disjoint from $N[\{1,2,3\}]$, and suppose that
\[
wx_1x_2x_3\in I^2.
\]
Since the support of $w$ is disjoint from $N[\{1,2,3\}]$, there are no
edges between $\operatorname{supp}(w)$ and $\{1,2,3\}$. Since
$wx_1x_2x_3$ is divisible by a product of two edges, at least one of
these two edges must be supported entirely on $\operatorname{supp}(w)$.
Consequently, $w$ is divisible by an edge of $G$, and hence
\(
w\in I.
\)
Therefore, Eq.~\eqref{eq_colon_triangle} follows. Consequently,
\[
\begin{aligned}
I^2:f
&=
\bigl(I^2:(x_1x_2x_3)\bigr):v\\
&=
\bigl(I+(x_j\mid j\in N[\{1,2,3\}])\bigr):v\\
&=
P+(x_j\mid j\in N[\{1,2,3\}]),
\end{aligned}
\]
where the last equality follows from the fact that
$v$ has support disjoint from $N[\{1,2,3\}]$.

In particular, $I^2:f$ is generated by variables and therefore is an associated prime of $I^2$. Since
$ 
\deg(f)\le\deg(u),
$ 
the conclusion follows.
\end{proof}


It is easy to see that the inequality in the above theorem can hold as an equality; for example, take $G = K_3$. At the same time, Theorem~\ref{thm:bigger} yields a class of graphs for which $v(I(G)^{(2)}) > v(I(G)^2)$. In light of these facts, we pose the following question:

\begin{question}
Classify all graphs $G$ such that $v(I(G)^{(2)}) = v(I(G)^2)$.
\end{question}

It is worth noting that Theorem~\ref{lem_min_sym_2} does not generalize to arbitrary squarefree monomial ideals, as shown in the following example:

\begin{example}
Let $R = \mathbb{Q}[x_1, \ldots, x_7]$ and 
\[
I = (x_1x_3, x_1x_5x_7, x_4x_6x_7, x_3x_4x_7, x_2x_3x_5x_6, x_1x_2x_4x_5) \subset R.
\]
Using the package \texttt{VNumber} in \textsf{Macaulay2}~\cite{M2}, one can verify that $v(I^{(2)}) = 4 < 5 = v(I^2)$.
\end{example}

\section{v-number of powers of edge ideals}\label{sec_exm}

In this section, we provide examples of graphs for which, for a fixed exponent $t\ge2$, the~difference
\[
v(I(G)^{(t)})-v(I(G)^t)
\]
can be made arbitrarily large as the number of variables increases. We then compute explicit formulas for the $v$-numbers of the ordinary and symbolic powers of edge ideals of paths and~cycles.

\begin{theorem}\label{thm:bigger}
Let $t\ge2$ be an integer. Then, for every integer $s\ge0$, there exists a graph $G$ such that
\[
v(I(G)^{(t)})-v(I(G)^t)\ge s-t.
\]
\end{theorem}

\begin{proof}
Let $G$ be the graph with vertex set
\[
V(G)=\{x_1,\ldots,x_{2t-1}\}\cup
\{y_{i,j}\mid i=1,\ldots,2t-1,\; j=1,\ldots,2s\},
\]
whose edges consist of the odd cycle $C_{2t-1}$ on $x_1,\ldots,x_{2t-1}$ together with the $4$-cycles
\[
x_i\,y_{i,2j-1}\,y_{i,2j}\,x_{i+1}
\]
for every $i=1,\ldots,2t-1$ and every $j=1,\ldots,s$, where we identify $x_{2t}$ with $x_1$. We denote by $I = I(G)$ the edge ideal of $G$. Observe that
$ 
I^t:(x_1\cdots x_{2t-1})=\mathfrak m.
$ 
Hence,
\[
v(I^t)\le 2t-1.
\]

Let $P$ be a minimal prime of $I$. We will prove that
\[
v_P(I^{(t)})\ge t+s-1.
\]
Let $f$ be a monomial satisfying
\( I^{(t)}:f=P.
\) Then
$ 
\deg_P(f)=t-1.
$ 
Note that $P$ corresponds to a minimal vertex cover of $G$, and hence covers the odd cycle $C_{2t-1}$. Since $C_{2t-1}$ is an odd cycle, there exists an index $r$ such that
$ 
x_r,x_{r+1}\in P.
$ 
Consequently,
\[
|P\cap\{y_{r,2j-1},y_{r,2j}\}|=1
\]
for every $j=1,\ldots,s$. Let $u_j$ denote the unique variable in this intersection, and let $v_j$
denote the other variable. Fix such a $j$, and let $Q$ be the prime ideal
obtained from $P$ by replacing $u_j$ with $v_j$. Then $Q$ corresponds to a
vertex cover of $G$, and hence
$ 
I\subseteq Q.
$ 
The vertex cover corresponding to $Q$ need not be minimal. Moreover,
$ 
u_jf\in I^{(t)}\subseteq Q^t.
$ 
Since $u_j\notin Q$, it follows that $f\in Q^t$. Therefore,
\[
\deg_Q(f)
=\deg_P(f)+\deg_{v_j}(f)-\deg_{u_j}(f)
\ge t.
\]
Since $\deg_P(f)=t-1$, we obtain
$ 
\deg_{v_j}(f)\ge1.
$ 
Hence, $v_j$ divides $f$ for every $j=1,\ldots,s$. Thus, $f$ is divisible by $s$ variables outside $P$. Since
$ 
\deg_P(f)=t-1,
$ 
we conclude that
\[
\deg(f)\ge (t-1)+s=t+s-1.
\]
Therefore,
\[
v_P(I^{(t)})\ge t+s-1,
\]
and hence
\[
v(I^{(t)})-v(I^t)\ge s-t,
\]
as desired.
\end{proof}

The construction in Theorem~\ref{thm:bigger} shows that the initial behavior of the $v$-function for powers of ideals can be quite arbitrary, even for edge ideals of graphs. More precisely, we have the following result:

\begin{theorem}
    For any given $k > 1$, there exists a graph $G$ such that 
    \[
    v(I(G)) < v(I(G)^2) < \dots < v(I(G)^{k-1}) > v(I(G)^k),
    \]
    and the $v$-function stabilizes at $k$. In other words, there exists a graph $G$ for which the $v$-function increases up to $k-1$, drops strictly at $k$, and then stabilizes.
\end{theorem}

\begin{proof}
    Fix $k \ge 2$. Consider the graph defined in Theorem~\ref{thm:bigger} with $t = k$ and $s$ sufficiently large (for instance, $s > 3k$), so that $v(I(G)) > 2k$. Note that $G$ has a unique odd cycle of length $2k-1$. Consequently, $I(G)^{(t)} = I(G)^t$ for all $t < k$. By \cite[Theorem 4.3]{YHC}, the sequence $v(I(G)^{(t)})$ is strictly increasing. On the other hand, $v(I(G)^t) = 2t - 1$ for all $t \ge k$. This completes the proof.
\end{proof}

Note that the $v$-numbers of powers of edge ideals of complete graphs follow from work of Biswas and Mandal \cite{BM}. As a simple consequence of our methods, we obtain a formula for the $v$-number of symbolic powers of edge ideals of complete graphs.

\begin{theorem}
Let \(K_n\) be a complete graph on \(n\ge2\) vertices. Then, for every \(t\ge1\), we have
\[
v\bigl(I(K_n)^{(t)}\bigr)= 
t+
\left\lceil\frac{t-1}{n-1}\right\rceil.
\]
\end{theorem}

\begin{proof}
For \(i=1,\ldots,n\), let
\[
P_i=(x_1,\ldots,\widehat{x_i},\ldots,x_n).
\]
Then \(P_1,\ldots,P_n\) are precisely the associated primes of \(I(K_n)\). First, we prove the upper bound. If \(t=1\), then
$
I(K_n):x_n=(x_1,\ldots,x_{n-1}),
$
and hence
$
v\bigl(I(K_n)\bigr)=1.
$

Now assume that \(t>1\). Write
$
t-1=(n-1)q+r
$
for some integers \(q,r\) with \(1\le r\le n-1\). Set
$
f=x_1^{q+1}\cdots x_r^{q+1}x_{r+1}^q\cdots x_{n-1}^q x_n^{q+2}.
$
Then
\[
\deg(f)
=(q+1)r+q(n-1-r)+q+2
=q(n-1)+r+q+2
=t+q+1.
\]
Moreover,
$
\deg_{P_n}(f)
=t+q+1-(q+2)
=t-1,
$ while
$
\deg_{P_i}(f)\ge t
$
for all \(i<n\). Therefore, by Lemma~\ref{lem_criterion},
$
I(K_n)^{(t)}:f=P_n.
$
It follows that
\[
v\bigl(I(K_n)^{(t)}\bigr)
\le t+q+1 = t+\left\lceil\frac{t-1}{n-1}\right\rceil.
\]

For the reverse inequality, let
$ 
f=x_1^{a_1}\cdots x_n^{a_n}
$ 
be a monomial such that
$ 
I(K_n)^{(t)}:f=P_n.
$ 
By Lemma~\ref{lem_criterion}, we have
$$
\deg_{P_n}(f)=t-1 \text{ and } \deg_{P_i}(f)\ge t
$$
for all \(i<n\). Let
$ 
d=a_1+\cdots+a_n=\deg(f).
$ 
These conditions are equivalent to
$ 
d-a_n=t-1\qquad \text{and} \qquad
d-a_i\ge t
$ 
for all \(i<n\). Summing the latter inequalities together with the equality for \(i=n\), we obtain
$ 
(n-1)d
\ge
(n-1)t+t-1.
$ 
Hence,
\[
d
\ge
t+\left\lceil\frac{t-1}{n-1}\right\rceil.
\]
The conclusion follows.
\end{proof}

We now turn to computing the $v$-numbers of ordinary and symbolic powers of edge ideals of paths and cycles. We first recall the following result due to Banerjee \cite{B}.

\begin{lemma}\label{lem_even}
Let \(G\) be a simple graph, and let \(I(G)\) denote its edge ideal. Let \(e_1,\ldots,e_{t-1}\) be edges of \(G\). Then
\[
I(G)^t:(e_1\cdots e_{t-1})
\]
is generated by monomials of degree two. Moreover, if \(uv\in I(G)^t:(e_1\cdots e_{t-1})\) then there exists an even walk
\[
x_0\cdots x_{2s+1}
\]
such that \(u=x_0\), \(v=x_{2s+1}\), and \(x_{2i+1}x_{2i+2}\) is one of the edges \(e_1,\ldots,e_{t-1}\) for each \(i\).
\end{lemma}
\begin{theorem}\label{thm_power_paths}
Let $P_n$ be a path on $n\ge2$ vertices. Then, for all $t\ge1$, we have
\[
v(I(P_n)^t)
= 2t-1+\max\left \{0,\left\lceil\frac{n-2t-3}{4}\right\rceil\right\}.
\]
\end{theorem}

\begin{proof}
First, we prove the upper bound. Let
\[
r=\max\left\{0,\left\lceil\frac{n-2t-3}{4}\right\rceil\right\}
\quad \text{and} \quad 
s=\min\left \{ t,\left\lfloor\frac{n}{2}\right\rfloor-1\right\}.
\]
If $n-2t\equiv 3 \pmod{4}$, set
\[
f=(x_3x_4)^{t-s+1}\cdots x_{2s-1}x_{2s}x_{2s+1}
x_{2s+5}\cdots x_{2s+1+4r}.
\]
Then we claim that
\[
I(P_n)^t:f=(x_2,x_4,\ldots,x_{n-1})=P.
\]
By definition,
\[
\deg(f)=2(t-s)+2s+r+1=2t-1+r,
\]
and
\[
\deg_P(f)=(t-s)+(s-1)=t-1.
\]
By Lemma~\ref{lem_colon}, it suffices to prove that
\[
x_{2j}\in I(P_n)^t:f
\]
for every $j=1,\ldots,\frac{n-1}{2}$.

Indeed,
\[
x_2f
=(x_2x_3)(x_3x_4)^{t-s+1}\cdots
(x_{2s}x_{2s+1})
(x_{2s+5}\cdots x_{2s+1+4r})\in I^t.
\]
For $2\le j\le s+1$, we have
\[
x_{2j}f
=(x_3x_4)^{t-s+1}\cdots
(x_{2j-1}x_{2j})
(x_{2j}x_{2j+1})\cdots
(x_{2s}x_{2s+1})
(x_{2s+5}\cdots x_{2s+1+4r})
\in I^t.
\]
For $s+1<j\le s+1+2r$, either $x_{2j-1}$ or $x_{2j+1}$ is one of
$ 
x_{2s+5},\ldots,x_{2s+1+4r},
$ 
and hence
\[
x_{2j}f\in I^t.
\]

Similarly, if
\( n-2t\equiv 0,1,2\pmod{4},
\) we set 
\[
f=(x_2x_3)^{t-s+1}(x_4x_5)\cdots
(x_{2s-2}x_{2s-1})x_{2s}
x_{2s+4}\cdots x_{2s+4r}.
\]
A similar argument shows that 
\[
I^t:f=(x_{2i-1}\mid 1\leq i \leq \frac{n+1}{2}).
\]

Since $I^t:f$ contains a variable, we have
$ 
\deg(f)\ge 2t-1.
$ 
Hence, it remains to prove the lower bound, and we may assume that
\(
n>2t+3.
\)
Let $f$ be a monomial such that
\[
I^t:f=P=(x_j\mid j\in C)
\]
for some $C\subseteq[n]$. In particular,
$ 
x_jf\in I^t
$ 
for every $j\in C$. Hence, $f\in I^{t-1}$. Thus, we may write
\[
f=e_1\cdots e_{t-1}g
\]
for some edges $e_i$ of $P_n$ and a monomial $g$.

We have
\[
I^t:f
= (I^t:e_1\cdots e_{t-1}):g = J:g,
\]
where
$ 
J=I^t:(e_1\cdots e_{t-1})
$ 
is the edge ideal of a simple graph $H$ obtained by adjoining even connections through $e_1,\ldots,e_{t-1}$.

For any $s\in\operatorname{supp}(g)$, let $L(s)$ and $R(s)$ denote the left and right extensions of $s$, respectively. More precisely,
\[
L(s)={i,\ldots,s},
\]
where $i$ is as small as possible such that $x_ix_s\in J$; that is, $x_i$ is the leftmost variable that can be reached from $x_s$ by an even connection through $e_1,\ldots,e_{t-1}$. Similarly,
\[
R(s)={s,\ldots,j},
\]
where $j$ is as large as possible such that $x_sx_j\in J$.

We denote by
\[
B(s)=L(s)\cup R(s)
\]
the block associated to $s$. Let $E_l(s)$ denote the set of edges among $e_1,\ldots,e_{t-1}$ that appear in the even connection from $s$ to the leftmost variable in $L(s)$, and let $E_r(s)$ denote the set of edges that appear in the even connection from $s$ to the rightmost variable in $R(s)$. Finally, set
\[
E(s)=E_l(s)\cup E_r(s).
\]

We claim the following.

\noindent\emph{Claim 1.} If $u<v$ are in $\operatorname{supp}(g)$ and
$ 
E_l(u)\cap E_l(v)\neq\emptyset,
$ 
then
$ 
B(u)\subseteq B(v).
$ 

\noindent\emph{Claim 2.} If $u<v$ are in $\operatorname{supp}(g)$ and
$ 
E_r(u)\cap E_r(v)\neq\emptyset,
$ 
then
$ 
B(v)\subseteq B(u).
$ 

\noindent\emph{Claim 3.} If $u<v$ are in $\operatorname{supp}(g)$, then
$
E_r(u)\cap E_l(v)=\emptyset.
$

It follows from these claims that the blocks $B(s)$ can be decomposed into maximal blocks
\[
B(s_1)\cup\cdots\cup B(s_m)
\]
such that the sets $E(s_j)$ are pairwise disjoint. We now prove the desired bound assuming these claims.

We may order the blocks $B(s_1),\ldots,B(s_m)$ lexicographically according to their initial points. Thus, write
$ 
B(s_i)=[a_i,b_i]
$ 
with $a_i<b_i$, and assume that
$ 
a_1\le\cdots\le a_m.
$ 
Since every element of $C$ is covered by some $B(s_i)$, we have
\[
a_{i+1}-b_i\le1,
\qquad
a_1-1\le1,
\qquad\text{and}\qquad
n-b_m\le1.
\]
Furthermore,
\[
|B(s_i)|=2q_i+3, 
\]
where $q_i$ is the number of edges in $E(s_i)$ used in $B(s_i)$. Thus,
\[
n
\le
2\sum_{i=1}^m q_i+3m+m+1
\le
2(t-1)+4\deg(g)+1.
\]
Hence,
\[
\deg(g)\ge
\left\lceil\frac{n-2t+1}{4}\right\rceil.
\]
Since
$ 
\deg(f)=2t-2+\deg(g),
$ 
this gives the desired lower bound. It remains only to prove Claims 1--3.

\noindent\emph{Proof of Claim 1.}
Assume that \(u<v\) are in \(\operatorname{supp}(g)\) and that
$ 
E_l(u)\cap E_l(v)\neq\emptyset.
$ 
Since any further left extension using one of the common edges is the same for both \(u\) and \(v\), we have
\[
L(u)\subseteq L(v).
\]
Note that if two vertices \(i\) and \(j\) are connected by an even connection, then (i-j) is odd, since the resulting graph is still bipartite with the usual parity bipartition. Thus, if \(u\) and \(v\) share an edge in their left extensions, then \(u\) and \(v\) have the same parity.

If the right extension of \(u\) does not pass \(v\), then clearly
\(
R(u)\subseteq L(v),
\)
and hence
\(
B(u)\subseteq B(v).
\)
On the other hand, if the right extension of (u) passes (v), then (v) can also be extended to the right using the same edges. Hence, the furthest reachable vertices to the right from \(u\) and \(v\) are the same. Therefore,
\[
B(u)\subseteq B(v).
\]

\noindent\emph{Proof of Claim 2.}
The proof is similar to that of Claim 1.

\noindent\emph{Proof of Claim 3.}
Assume, by contradiction, that \(u<v\) are in \(\operatorname{supp}(g)\) and
$ 
E_r(u)\cap E_l(v)\neq\emptyset.
$ 
In other words, there exists an edge \(e\) that can be reached from \(u\) by extending to the right and from \(v\) by extending to the left. Concatenating the edges in the right extension from \(u\) to \(e\) with the edges in the left extension from \(v\) to \(e\) gives an even connection from \(u\) to \(v\). Hence,
$ 
x_ux_v\in J.
$ 
It follows that
\(g\in J,
\)
which contradicts \(J:g=P\), since \(P\) is a proper ideal. Therefore,
\[
E_r(u)\cap E_l(v)=\emptyset.
\]
This proves Claim 3.
\end{proof}

\begin{theorem}
Let $C_n$ be a cycle on $n\ge 3$ vertices. Then, for all $t\ge1$, we have
\[
v(I(C_n)^t)=2t-1+\max\left\{0,\left\lceil\frac{n-2t-2}{4}\right\rceil\right\}.
\]
\end{theorem}

\begin{proof} First, we prove the upper bound. Let
\[
r = \max\left\{0, \left\lceil\frac{n-2t-2}{4}\right\rceil\right\}
\quad \text{and} \quad 
s = \min\left\{t, \left\lfloor\frac{n}{2}\right\rfloor - 1\right\}.
\]
If $n - 2t \equiv 3 \pmod 4$, set 
\[
f = (x_1x_2)^{t-s+1} (x_3x_4) \cdots (x_{2s-3} x_{2s-2}) \cdot g,
\]
where $g = x_{2s-1}x_{2s+3}\cdots x_{2s-1+4r-4} x_{2s-1+4r-2}$. Then
\[
\deg(f) = 2t - 1 + r.
\]
We claim that $I^t : f = P = (x_2, x_4, \ldots, x_{n-1}, x_n)$.

We have $\deg_P(f) = t - 1$. By Lemma~\ref{lem_colon}, it suffices to show that $x_n \in I^t : f$ and $x_{2j} \in I^t : f$ for all $j = 1, \ldots, \frac{n-1}{2}$. Indeed, for every $1 \le j \le s-1$, we have
\[
x_{2j}f = (x_1x_2)^{t-s+1} \cdots (x_{2j-1}x_{2j}) (x_{2j}x_{2j+1}) \cdots (x_{2s-2}x_{2s-1}) x_{2s+3} \cdots x_{2s-1+4r-4} x_{2s-1+4r-2} \in I^t.
\]
On the other hand, if $s \le j \le \frac{n-1}{2}$, then $x_{2j}g \in I$, and hence $x_{2j}f \in I^t$. Furthermore, 
\[
x_n f = (x_1x_2)^{t-s} (x_nx_1) (x_2x_3) \cdots (x_{2s-2}x_{2s-1}) x_{2s+3} \cdots x_{2s-1 + 4r-4} x_{2s-1 + 4r-2} \in I^t.
\]

If $n - 2t \equiv 0, 1, 2 \pmod 4$, set 
\[
f = (x_1x_2)^{t-s+1} (x_3x_4) \cdots (x_{2s-3} x_{2s-2}) \cdot g,
\]
where $g = x_{2s-1}x_{2s+3}\cdots x_{2s-1+4r}$. Then a similar argument shows that 
\[
I^t : f = 
\begin{cases} 
(x_2, x_4, \ldots, x_n) & \text{if } n \text{ is even,} \\ 
(x_2, x_4, \ldots, x_{n-1}, x_n) & \text{if } n \text{ is odd.} 
\end{cases}
\]

Thus, it remains to prove the lower bound. Since $I^t:f$ contains a variable, we have
\[
\deg(f)\ge 2t-1.
\]
Therefore, we may assume that
\( 
n>2t+2.
\) In this case, we have
\(
I^t=I^{(t)}.
\)
Hence, all associated primes of $I^t$ are minimal. Let
$ 
P=(x_i\mid i\in C)
$ 
be a minimal prime of $I(C_n)$, and let $f$ be a monomial satisfying
\(
I^t:f=P.
\) 
We need to prove that
\[
\deg(f)\ge 2t-1+r.
\]
The argument is similar to that of the proof of Theorem~\ref{thm_power_paths}. Since $x_j f \in I^t$ for every $j \in C$, we deduce that $f \in I^{t-1}$. Thus, we may write $f = e_1 \cdots e_{t-1} g$ for some edges $e_i$ of $C_n$ and a monomial $g$ of degree at least $1$. We have 
$$I^t : f = (I^t : e_1 \cdots e_{t-1}) : g = J : g,$$
where $J = I^t : (e_1 \cdots e_{t-1})$ is the edge ideal of a graph $H$ obtained by adjoining even connections through $e_1, \ldots, e_{t-1}$. 

As in the proof of Theorem~\ref{thm_power_paths}, for each $s\in\operatorname{supp}(g)$, let $L(s)$ and $R(s)$ denote the counterclockwise and clockwise extensions of $s$, respectively, and set $B(s)=L(s) \cup R(s)$. Since $B(s)$ contains at most $2q+3$ vertices, where $q$ is the number of edges among $e_1,\ldots,e_{t-1}$ used in the extensions defining $B(s)$, and since $n\ge 2t+3$, it follows that $B(s)$ is a proper interval of the cycle. Thus, locally, each of these blocks behaves similarly to the corresponding blocks in the path case. Consequently, the three claims from Theorem~\ref{thm_power_paths} carry over, allowing us to decompose the blocks $B(s)$ into maximal blocks, $B(s_1) \cup \cdots \cup B(s_m)$ such that the sets $E(s_j)$ are pairwise disjoint. 

We order the blocks $B(s_1), \ldots, B(s_m)$ in the clockwise direction, and write $B(s_i) = [a_i, b_i]$, where we identify a vertex $j$ with $j - n$ if $j > n$. Since every element of $C$ is covered by some $B(s_i)$, we obtain the inequalities 
$$a_{i+1} - b_i \le 1 \quad \text{and} \quad n + a_1 - b_m \le 1.$$ 
Furthermore, $|B(s_i)| = 2q_i + 3$, where $q_i$ is the number of edges in $E(s_i)$ used in $B(s_i)$. Thus,
$$n \le 2 \sum_{i=1}^m q_i + 3m + m \le 2(t-1) + 4 \deg(g).$$
Hence, $\deg(g) \ge \left\lceil \frac{n-2t+2}{4} \right\rceil$. Since $\deg(f) = 2t - 2 + \deg(g)$, the conclusion follows.
\end{proof}

\begin{theorem}
Let $I=I(C_{2m-1})$, where $m\geq 2$. Then, for every $t\geq m$,
\[
v\bigl(I^{(t)}\bigr)
=
\left\lceil\frac{(2m-1)t-1}{m}\right\rceil.
\]
Equivalently, if $t=km+r$, where $k\geq 0$ and $1\leq r\leq m$, then
\[
v\bigl(I^{(t)}\bigr)
=
\begin{cases}
k(2m-1)+2r, & 1\leq r\leq m-2,\\
k(2m-1)+2r-1, & m-1\leq r\leq m.
\end{cases}
\]
\end{theorem}

\begin{proof}
Set $n=2m-1$, and label the vertices of $C_n$ by
$x_1,\ldots,x_n$. Let $C_1,\ldots,C_n$ be the minimum vertex covers
of $C_n$. Each $C_i$ has cardinality $m$, and every vertex belongs
to exactly $m$ of these covers. Hence, for every monomial $u$,
\begin{align}
\sum_{i=1}^{n}\deg_{C_i}(u)=m\deg(u). \label{eq.111}
\end{align}

Suppose that $I^{(t)}:u=P_D$ for some minimal vertex cover $D$. By
Lemma \ref{lem_deg}, $\deg_D(u)=t-1$. If $C_i\neq D$, then
$D\setminus C_i\neq\emptyset$. Choose $x_j\in D\setminus C_i$.
Since $x_ju\in I^{(t)}\subseteq P_{C_i}^t$ and
$x_j\notin P_{C_i}$, it follows that $\deg_{C_i}(u)\geq t$.
Consequently, among the numbers $\deg_{C_i}(u)$, at most one is
equal to $t-1$, while all the others are at least $t$. By (\ref{eq.111}), $m\deg(u)\geq nt-1$, and therefore
\begin{align}
v\bigl(I^{(t)}\bigr) \geq \left\lceil\frac{nt-1}{m}\right\rceil.
\label{eq.Lowerbound}
\end{align}

We now construct monomials attaining this bound. Fix the minimum
vertex cover $C=\{x_1,x_3,\ldots,x_{2m-1}\}$ and set
\[
f_{m-1}
=
x_{2m-2}\prod_{j=1}^{m-2}x_{2j}x_{2j+1}.
\]
Then $\deg_C(f_{m-1})=m-2$. Moreover, $x_1f_{m-1}$ is divisible by
\[
(x_1x_2)(x_3x_4)\cdots(x_{2m-3}x_{2m-2}),
\]
and $x_{2m-1}f_{m-1}$ is divisible by
\[
(x_2x_3)(x_4x_5)\cdots(x_{2m-2}x_{2m-1}).
\]
For $1\leq q\leq m-2$, one has
\[x_{2q+1}f_{m-1}
=
\left(\prod_{j=1}^{q-1}x_{2j}x_{2j+1}\right)
(x_{2q}x_{2q+1})(x_{2q+1}x_{2q+2}) \cdot
\left(\prod_{j=q+1}^{m-2}x_{2j+1}x_{2j+2}\right).
\]
Thus $x_if_{m-1}\in I^{m-1}$ for every $x_i\in C$. It follows from Lemma \ref{lem_colon} that $I^{m-1}:f_{m-1}=P_C$.

For every $s\geq m-1$, define
$f_s=f_{m-1}(x_1x_2)^{s-m+1}$. Then $x_if_s\in I^s$ for every
$x_i\in C$, while $\deg_C(f_s)=s-1$. Hence, again by Lemma \ref{lem_colon},
\begin{align}
I^s:f_s=P_C. \label{eq.112}
\end{align}
Notice also that $\deg(f_s)=2s-1$.

Let $c=x_1\cdots x_n$. Since every minimal vertex cover of $C_n$ has at least $m$ vertices, $c\in I^{(m)}$.
Recall that $t=km+r$, where $1\leq r\leq m$. 

Suppose first that
$1\leq r\leq m-2$. Since $t\geq m$, we have $k\geq 1$. Set
$u=c^{k-1}f_{m+r}$. For every $x_i\in C$, (\ref{eq.112}) gives
$x_if_{m+r}\in I^{m+r}$, and hence $x_iu\in I^{(t)}$. Furthermore,
\[
\deg_C(u)=(k-1)m+(m+r-1)=t-1.
\]
Lemma \ref{lem_colon} therefore gives $I^{(t)}:u=P_C$, and
\begin{align}
\deg(u) = (k-1)(2m-1)+2(m+r)-1 = k(2m-1)+2r.
\label{eq.113}
\end{align}

Now suppose that $m-1\leq r\leq m$, and set $u=c^kf_r$. For every
$x_i\in C$, we have $x_if_r\in I^r$, so $x_iu\in I^{(t)}$.
Moreover, $\deg_C(u)=km+r-1=t-1$. Thus
$I^{(t)}:u=P_C$, and
\begin{align}
\deg(u)=k(2m-1)+2r-1.
\label{eq.114}
\end{align}

Finally, a direct computation gives
\[
\left\lceil\frac{(2m-1)t-1}{m}\right\rceil
=
\begin{cases}
k(2m-1)+2r, & 1\leq r\leq m-2,\\
k(2m-1)+2r-1, & m-1\leq r\leq m.
\end{cases}
\]
The upper bounds furnished by (\ref{eq.113}) and (\ref{eq.114}), together with the lower
bound (\ref{eq.Lowerbound}), complete the~proof.
\end{proof}


\end{document}